\documentclass[11pt]{article}
\pdfoutput=1
\usepackage{amsfonts}
\usepackage{amsmath}
\usepackage{enumerate}
\usepackage[parfill]{parskip}
\usepackage{color}

\newcommand{\nv}{\boldsymbol}
\newcommand{\lc}{[\![}
\newcommand{\rc}{]\!]}
\newcommand{\1}[1]{{\boldsymbol 1_{\{#1\}}}}
\newtheorem{theorem}{Theorem}

\newtheorem{assumption}{Assumption}

\newtheorem{corollary}{Corollary}

\newtheorem{definition}{Definition}

\newtheorem{lemma}{Lemma}

\newenvironment{proof}[1][Proof]{\textbf{#1.} }{\ \rule{0.5em}{0.5em}}
\begin{document}

\title{Linking progressive and initial filtration expansions}
\author{Younes Kchia\thanks{%
Centre de Math\'ematiques Appliqu\'ees, Ecole Polytechnique, Paris} \and %
Martin Larsson\thanks{%
School of Operations Research, Cornell University, Ithaca, NY, 14853} \and %
Philip Protter\thanks{%
Statistics Department, Columbia University, New York, NY, 10027} 
\thanks{%
Supported in part by NSF grant DMS-0906995}} 
\date{\today}
\maketitle

\begin{abstract}
In this paper we study progressive filtration expansions with random times. We show how semimartingale decompositions in the expanded filtration can be obtained using a natural link between progressive and initial expansions. The link is, on an intuitive level, that the two coincide after the random time. We make this idea precise and use it to establish known and new results in the case of expansion with a single random time. The methods are then extended to the multiple time case, without any restrictions on the ordering of the individual times. Finally we study the link between the expanded filtrations from the point of view of filtration shrinkage. As the main analysis progresses, we indicate how the techniques can be generalized to other types of expansions.
\end{abstract}

\section{Introduction}

Expansion of filtrations is a well-studied topic that has been investigated both in theoretical and applied contexts, see for instance~\cite{Jacod:1987}, \cite{Jeanblanc/LeCam:2009}, \cite{Jeulin/Yor:1985}. There are two main types of filtration expansion: initial expansion and progressive expansion. The initial expansion of a filtration $\mathbb F=(\mathcal F_t)_{t\geq 0}$ with a random variable $\tau$ is the filtration $\mathbb H$ obtained as the right-continuous modification of $(\mathcal F_t\vee\sigma(\tau))_{t\geq 0}$. \emph{A priori} there is no particular interpretation attached to this random variable. The progressive expansion $\mathbb G$ is obtained as any right-continuous filtration containing $\mathbb F$ and making $\tau$ a stopping time. In this case $\tau$ should of course be nonnegative since it has the interpretation of a random time. We often use the smallest such filtration in this paper.  On a complete probability space $(\Omega,\mathcal{F},P)$ a filtration $\mathbb{F}$ is said to satisfy the ``usual hypotheses'' if it is right continuous and if $\mathcal{F}_0$ contains all the $P$ null sets of $\mathcal{F}$.  Therefore $\mathcal{F}_t$ contains all the $P$ null sets of $\mathcal{F}$ as well, for any $t\geq 0$.  We will assume throughout this paper that all filtrations satisfy the usual hypotheses.

To the best of our knowledge, these two types of expansions have so far been viewed as inherently different in the literature. The purpose of the present paper is to demonstrate that this is not the case---in fact, there is a very natural connection between the initial and progressive expansions. The reason is, on an intuitive level, that the filtrations $\mathbb G$ and $\mathbb H$ coincide after time $\tau$. We make this idea precise for filtrations $\mathbb H$ that are not necessarily obtained as initial expansions. This, in combination with a classical theorem by Jeulin and Yor, allows us to show how the semimartingale decomposition of an $\mathbb F$~local martingale, when viewed in the progressively expanded filtration $\mathbb G$, can be obtained on all of $[0,\infty)$, provided that its decomposition in the filtration $\mathbb H$ is known. One well-known situation where this is the case is when \emph{Jacod's criterion} is satisfied. This is, however, not the only case, and we give an example using techniques based on Malliavin calculus developed by Imkeller et al.~\cite{Imkeller_etal:2001}. These developments, which all concern expansion with a single random time, are treated in Section~\ref{S:1}. The technique is, however, applicable in more general situations than expansion with a single random time. As an indication of this, we perform \emph{en passant} the same analysis for what we call the $(\tau,X)$-progressive expansion of $\mathbb F$, denoted $\mathbb G^{(\tau,X)}$. Here $\tau$ becomes a stopping time in the larger filtration, and the random variable $X$ becomes $\mathcal G^{(\tau,X)}_\tau$-measurable.

In Section~\ref{S:2} we extend these ideas in order to deal with the case where the base filtration~$\mathbb F$ is expanded progressively with a whole vector $\nv \tau=(\tau_1,\ldots,\tau_n)$ of random times. Unlike previous work in the literature, see for instance \cite{Jeanblanc/LeCam:2009}, we do not impose any conditions on the ordering of the individual times. After establishing a general semimartingale decomposition result we treat the special case where Jacod's criterion is satisfied for the whole vector $\nv\tau$, and we show how the decompositions may be expressed in terms of $\mathbb{F}$~conditional densities of $\nv \tau$ with respect to its law.

Finally, in Section~\ref{S:shr} we take a different point of view and study the link between the filtrations $\mathbb G$ and $\mathbb H$ from the perspective of filtration shrinkage.

\section{Expansion with one random time} \label{S:1}

Assume that a filtered probability space $(\Omega, \mathcal F, \mathbb F, P)$ is given, and let $\tau$ be a \emph{random time}, i.e.~a nonnegative random variable. Typically $\tau$ is not a stopping time with respect to $\mathbb F$. Consider now the larger filtrations $\mathbb G^{\tau} = (\mathcal G^{\tau}_t)_{t\geq 0}$ and $\tilde{\mathbb{G}} = (\tilde{\mathcal G_t})_{t\geq 0} $ given by
$$
\mathcal G^{\tau}_t = \bigcap_{u>t} \mathcal G^{0,\tau}_u  \qquad \text{where}\qquad   \mathcal G^{0,\tau}_t =\mathcal F_t \vee \sigma(\tau \wedge t).
$$
and
$$
\tilde{\mathcal{G}_t}=\big\{A \in \mathcal{F}, \exists A_t\in \mathcal{F}_t \mid A\cap\{\tau>t\} = A_t\cap\{\tau>t\} \big\}
$$
One normally refers to $\mathbb G^{\tau}$ as the \emph{progressive expansion} of $\mathbb F$ with $\tau$, and it can be characterized as the smallest right-continuous filtration that contains $\mathbb F$ and makes $\tau$ a stopping time.

Throughout this section $\mathbb{G}$ will denote any right-continuous filtration containing $\mathbb{F}$, making $\tau$ a stopping time and satisfying $\mathcal{G}_t\cap\{\tau>t\} = \mathcal{F}_t\cap\{\tau>t\}$ for all $t\geq 0$. Our goal in this section is to analyze how $\mathbb F$~semimartingales behave in the progressively expanded filtration $\mathbb{G}$. In particular, in case they remain semimartingales in $\mathbb G$, we are interested in their canonical decompositions. Under a well known and well studied hypothesis due to Jacod~\cite{Jacod:1987}, this has been done by Jeanblanc and Le~Cam~\cite{Jeanblanc/LeCam:2009} in the filtration $\mathbb{G}^{\tau}$. As one consequence of our approach, we are able to provide a short proof of their main result. Moreover, our technique also works for a larger class of progressively expanded filtrations and under other conditions than Jacod's criterion.

The $\mathbb G$~decomposition before time $\tau$ of an $\mathbb F$~local martingale $M$ follows from a classical and very general theorem by Jeulin and Yor~\cite{Jeulin/Yor:1985}, which we now recall.

\begin{theorem}\label{T:JY}
Fix an $\mathbb F$~local martingale $M$ and define $Z_t=P(\tau>t\mid\mathcal F_t)$ as the optional projection of $\nv 1_{\lc \tau,\infty\lc}$ onto $\mathbb F$, let $\mu$ be the martingale part of its Doob-Meyer decomposition, and let $J$ be the dual predictable projection of $\Delta M_\tau\nv 1_{\lc\tau,\infty\lc}$ onto~$\mathbb F$. Then
$$
M_{t\wedge\tau} - \int_0^{t\wedge\tau} \frac{d\langle M,\mu \rangle_s + dJ_s}{Z_{s-}}
$$
is a local martingale in both $\mathbb G^{\tau}$ and $\tilde{\mathbb G}$.
\end{theorem}

The $\mathbb{G}$ decomposition before time $\tau$ follows as a straightforward corollary of Theorem \ref{T:JY} using the following shrinkage result by F\"{o}llmer and Protter, see \cite{Follmer/Protter:2010}.

\begin{lemma}\label{L:shr}
Let $\mathbb{E}\subset\mathbb{F}\subset\mathbb{G}$ be three filtrations. Let $X$ be a $\mathbb{G}$ local martingale. If both $X$ and the optional projection of $X$ onto $\mathbb{E}$ are local martingales, then the optional projection of $X$ onto $\mathbb{F}$ is an $\mathbb{F}$ local martingale.
\end{lemma}

Putting Theorem \ref{T:JY} and Lemma \ref{L:shr} together provides the $\mathbb{G}$~decomposition before time $\tau$ of an $\mathbb{F}$ local martingale $M$ as given in the following theorem.

\begin{theorem}\label{T:JYExt}
Fix an $\mathbb{F}$ local martingale $M$. Then 
$$
M_{t\wedge\tau} - \int_0^{t\wedge\tau} \frac{d\langle M,\mu \rangle_s + dJ_s}{Z_{s-}}
$$
is a $\mathbb{G}$ local martingale. Here, the quantities $M$, $\mu$ and $J$ are defined as in Theorem \ref{T:JY}.
\end{theorem}

Finding the decomposition after $\tau$ is more complicated, but it can be obtained provided that it is known with respect to a suitable auxiliary filtration $\mathbb H$. More precisely, one needs that $\mathbb H$ and $\mathbb G$ coincide after $\tau$, in a certain sense. One such filtration $\mathbb H$ when $\mathbb G$ is taken to be $\mathbb{G}^{\tau}$ is, as we will see later, the initial expansion of $\mathbb F$ with $\tau$. We now make precise what it means for two filtrations to  coincide after $\tau$.

\begin{definition}
Let $\mathbb G$ and $\mathbb H$ be two filtrations such that $\mathbb G\subset\mathbb H$, and let $\tau$ be an $\mathbb H$~stopping time. Then $\mathbb G$ and $\mathbb H$ are said to \emph{coincide after~$\tau$} if for every $\mathbb H$~adapted process $X$, the process
$$
\nv 1_{\lc \tau,\infty\lc}(X - X_\tau)
$$
is $\mathbb G$~adapted.
\end{definition}

The following lemma establishes some basic properties of filtrations that coincide after~$\tau$.

\begin{lemma} \label{L:20}
Assume that $\mathbb G$ and $\mathbb H$ are right-continuous and coincide after $\tau$. Then
\begin{itemize}
\item[(i)] For every $\mathbb H$~stopping time $T$, $T\vee\tau$ is $\mathbb G$~stopping time. In particular, $\tau$ itself is a $\mathbb G$~stopping time.
\item[(ii)] For every $\mathbb H$~optional (predictable) process $X$, the process $\nv 1_{\lc \tau,\infty\lc}(X - X_\tau)$ is $\mathbb G$~optional (predictable).
\end{itemize}
\end{lemma}

\begin{proof}
For $(i)$, let $T$ be an $\mathbb H$~stopping time. Then $T\vee \tau$ is again an $\mathbb H$~stopping time, so $X=\nv 1_{\lc0,T\vee\tau\rc}$ is $\mathbb H$~adapted. Thus 
$$
\nv 1_{\lc \tau,\infty\lc}(r)(X_r - X_\tau) = -\1{T\vee\tau < r}
$$
is $\mathcal G_r$-measurable by hypothesis. This holds for every $r\geq 0$, so $T\vee\tau$ is a $\mathbb G$~stopping time since $\mathbb G$ is right-continuous.

For $(ii)$, let $X$ be of the form $X = h\nv 1_{\lc s,t\lc}$ for an $\mathcal H_s$-measurable random variable $h$ and fixed $0\leq s<t$. Then
$$
Y_r = \nv 1_{\lc \tau,\infty\lc}(r)(X_r-X_{\tau}) = \1{\tau\leq r} \big( h\1{s\leq r < t} - h\1{s\leq\tau<t}\big)
$$
is $\mathcal G_r$-measurable by assumption and defines a c\`adl\`ag process. Hence $Y$ is $\mathbb G$~optional, and the Monotone Class theorem implies the claim for $\mathbb H$~optional processes. The predictable case is similar.
\end{proof}

Before giving the first result on the $\mathbb G$~semimartingale decomposition of an $\mathbb F$~local martingale, under the general assumption that such a decomposition is available in some filtration $\mathbb H$ that coincides with $\mathbb G$ after~$\tau$, we provide examples of such pairs of filtrations $(\mathbb{G}, \mathbb{H})$. For the progressively expanded filtration $\mathbb{G}^{\tau}$, it turns out that the filtration $\mathbb H$ given as the initial expansion of $\mathbb F$ with $\tau$ coincides with $\mathbb G^{\tau}$ after $\tau$. Recall that the initial expansion is defined as follows.
$$
\mathcal H_t = \bigcap_{u>t} \mathcal H^0_u  \qquad \text{where}\qquad   \mathcal H^0_t=\mathcal F_t\vee\sigma(\tau).
$$

\begin{lemma}\label{L:21}
Let $\mathbb H$ be the initial expansion of $\mathbb F$ with~$\tau$. Then $\mathbb G^{\tau}$ and $\mathbb H$ coincide after $\tau$.
\end{lemma}

\begin{proof}
Let $X = h\nv 1_{\lc s,t\lc}$ for an $\mathcal H_s$-measurable random variable $h$ and fixed $0\leq s<t$. Then
\begin{align*}
Y_r = \nv 1_{\lc \tau,\infty\lc}(r)(X_r-X_{\tau}) &=  \1{\tau\leq r} \big( h\1{s\leq r < t} - h\1{s\leq\tau<t}\big) \\
&= h\1{s\leq r < t}\1{\tau\leq r} - h\1{s\leq\tau<t}\1{t\leq r} - h\1{s\leq\tau\leq r}\1{r< t}.
\end{align*}
It is enough to prove that the three terms on the right side are $\mathcal G^{\tau}_r$-measurable. Let us consider the first term, the other ones being similar. First let $h$ be of the form $f k(\tau)$ for some $\mathcal F_s$-measurable $f$ and Borel function $k$. Then
$$
h\1{s\leq r < t}\1{\tau\leq r} = f k(\tau) \1{s\leq r < t}\1{\tau\leq r} = f k(r\wedge \tau) \1{s\leq r < t}\1{\tau\leq r},
$$
which is $\mathcal G^{\tau}_r$-measurable. Using the Monotone Class theorem the result follows for every $\mathcal H^0_s$-measurable $h$, and finally for every $\mathcal H_s$-measurable $h$ by a standard argument.
\end{proof}

As a corollary we may give an alternative characterization of the progressively enlarged filtration~$\mathbb G^{\tau}$ as the smallest right-continuous filtration that contains~$\mathbb F$ and coincides with $\mathbb H$ after~$\tau$, where $\mathbb H$ is the initial expansion of $\mathbb F$.

\begin{corollary}
Let $\mathbb H$ be the initial expansion of $\mathbb F$ with~$\tau$. Then
$$
\mathbb G^{\tau} = \bigcap \big\{ \tilde{\mathbb G} : \mathbb F\subset\tilde{\mathbb G}\subset \mathbb H, \ \tilde{\mathbb G}\ \text{is\ right-continuous,\ and\ } \tilde{\mathbb G}\ \text{and\ }\mathbb H \text{\ coincide\ after\ } \tau\big\}
$$
\end{corollary}

\begin{proof}
To show ``$\subset$'', let $\tilde{\mathbb G}$ be an arbitrary element in the class over which we take the intersection. By Lemma~\ref{L:20}, $\tilde{\mathbb G}$ is a right-continuous filtration that makes $\tau$ a stopping time and contains~$\mathbb F$. Hence $\mathbb G^{\tau}$ is included in each $\tilde{\mathbb G}$ and thus in their intersection. For ``$\supset$'', Lemma~\ref{L:21} implies that $\mathbb G^{\tau}$ coincides with $\mathbb H$ after~$\tau$. It is right-continuous, so it is one of the filtrations we are intersecting.
\end{proof}

Let $X$ be a random variable. Consider now the filtration $\mathbb{G}^{(\tau, X)}$, which we call the \emph{$(\tau,X)$-expansion of $\mathbb F$}, given by
$$
\mathcal G^{(\tau,X)}_t = \bigcap_{u>t} \mathcal G^{0,(\tau,X)}_u  \qquad \text{where}\qquad   \mathcal G^{0,(\tau,X)}_t =\mathcal F_t \vee \sigma(\tau \wedge t) \vee \sigma(X1_{\{\tau\leq t\}}).
$$
The filtration $\mathbb{G}^{(\tau, X)}$ is the smallest right-continuous filtration containing $\mathbb{F}$, that makes $\tau$ a stopping time and such that $X$ is $\mathcal{G}^{(\tau, X)}_{\tau}$~measurable. As in Lemma \ref{L:21} it is easy to prove the following result.
\begin{lemma}\label{L:22}
Let $\mathbb{H}$ be the initial expansion of $\mathbb{F}$ with $(\tau, X)$. Then $\mathbb{G}^{(\tau, X)}$ coincides with $\mathbb{H}$ after~$\tau$. 
\end{lemma}

Also, $\mathbb{G}^{(\tau,X)}$ satisfies the crucial condition $\mathcal{G}^{(\tau,X)}_t\cap\{\tau>t\}=\mathcal{F}_t\cap\{\tau>t\}$ for all $t\geq 0$.
\begin{lemma}\label{L:inters}
The progressively expanded filtration $\mathbb{G}^{(\tau,X)}$ satisfies
$$
\mathcal{G}^{(\tau,X)}_t\cap\{\tau>t\}=\mathcal{F}_t\cap\{\tau>t\}
$$
for all $t\geq 0$.
\end{lemma}

\begin{proof}
It is well known that $\mathcal{G}^{\tau}_t\cap\{\tau>t\}=\mathcal{F}_t\cap\{\tau>t\}$. Let $H_t = Y_th(X1_{\{\tau\leq t\}})$, where $Y_t$ is $\mathcal{F}_t$-measurable and bounded, and $h$ is a bounded Borel function. Then $H_t1_{\{\tau>t\}}=Y_th(0)1_{\{\tau>t\}}$ which is measurable with
respect to $\{\tau>t\}\cap\mathcal{G}^{\tau}_t=\{\tau>t\}\cap\mathcal{F}_t$. The Monotone Class Theorem now proves that
$\mathcal{G}^{(\tau,X)}_t\cap\{\tau>t\} \subset \mathcal{F}_t\cap\{\tau>t\}$. The reverse inclusion is clear.
\end{proof}

We are now ready to give the first result on the $\mathbb G$~semimartingale decomposition of an $\mathbb F$~local martingale, under the general assumption that such a decomposition is available in some filtration $\mathbb H$ that coincides with $\mathbb G$ after~$\tau$.

\begin{theorem} \label{T:gen}
Let $M$ be an $\mathbb F$~local martingale. Let $\mathbb G$ be any progressive expansion of $\mathbb F$ with $\tau$ satisfying $\mathcal{G}_t\cap\{\tau>t\} = \mathcal{F}_t\cap\{\tau>t\}$ for all $t\geq 0$ and let $\mathbb H$~be a filtration that coincides with $\mathbb G$ after~$\tau$. Suppose there exists an $\mathbb H$~predictable finite variation process $A$ such that $M-A$ is an $\mathbb H$~local martingale. Then $M$ is a $\mathbb G$~semimartingale, and
$$
M_t -  \int_0^{t\wedge\tau} \frac{d\langle M,\mu \rangle_s + dJ_s}{Z_{s-}} - \int_{t\wedge\tau}^t dA_s
$$
is the local martingale part of its $\mathbb G$~decomposition. Here $Z$, $\mu$ and $J$ are defined as in Theorem~\ref{T:JY}.
\end{theorem}

\begin{proof}
The process $M^{\mathbb G}_t = M_{t\wedge\tau} - \int_0^{t\wedge\tau} \frac{d\langle M,\mu \rangle_s + dJ_s}{Z_{s-}}$ is a $\mathbb G$~local martingale by the Jeulin-Yor theorem (Theorem~\ref{T:JYExt}). Next, define
$$
M^{\mathbb H} = \nv 1_{\lc\tau,\infty\lc}(\tilde M - \tilde M_\tau),
$$
where $\tilde M_t = M_t-A_t$. Since $\tilde M$ is an $\mathbb H$~local martingale, $M^{\mathbb H}$ is also. Moreover, if $(T_n)_{n\geq 1}$ is a sequence of $\mathbb H$~stopping times that reduce $\tilde M$, then $T_n' = T_n\vee\tau$ yields a reducing sequence for $M^{\mathbb H}$. Lemma~\ref{L:20}~$(i)$ shows that the $T_n'$ are in fact $\mathbb G$~stopping times, and since $\mathbb G$ and $\mathbb H$ coincide after $\tau$, $M^{\mathbb H}$ is $\mathbb G$~adapted. This implies that $M^{\mathbb H}_{\cdot\wedge T_n'}$ is an $\mathbb H$~martingale that is $\mathbb G$~adapted, and is therefore a $\mathbb G$~martingale. It follows that $M^{\mathbb H}$ is a $\mathbb G$~local martingale. It now only remains to observe that
$$
M^{\mathbb G}_t + M^{\mathbb H}_t = M_t - \int_0^{t\wedge\tau} \frac{d\langle M,\mu \rangle_s + dJ_s}{Z_{s-}} - \int_{t\wedge\tau}^t dA_s,
$$
which thus is a $\mathbb G$~local martingale. Finally, by Lemma~\ref{L:20}~$(ii)$, the last term is $\mathbb G$~predictable, so we obtain indeed the $\mathbb G$~semimartingale decomposition.
\end{proof}

Part of the proof of Theorem~\ref{T:gen} can be viewed as a statement about filtration shrinkage. According to a result by F\"ollmer and Protter~\cite{Follmer/Protter:2010}, if $\mathbb G\subset\mathbb H$ are two nested filtrations and $L$ is an $\mathbb H$~local martingale that can be reduced using $\mathbb G$~stopping times, then its optional projection onto~$\mathbb G$ is again a local martingale. In our case $L$ corresponds to $M^{\mathbb H}$, which is $\mathbb G$ adapted and hence coincides with its optional projection.

We now proceed to examine two particular situations where $\mathbb G$ and $\mathbb H$ coincide after $\tau$, and where the $\mathbb H$~decomposition $M-A$ is available. First we make an absolute continuity assumption on the $\mathcal F_t$~conditional laws of~$\tau$ or $(\tau, X)$, known as Jacod's criterion. We then assume that $\mathbb F$ is a Wiener filtration, and impose a condition related to the Malliavin derivatives of the process of $\mathcal F_t$~conditional distributions. This is based on theory developed by Imkeller, Pontier and Weisz~\cite{Imkeller_etal:2001} and Imkeller~\cite{Imkeller:2003}.

\subsection{Jacod's criterion}

In this section we study the case where $\tau$ or $(\tau, X)$ satisfy \emph{Jacod's criterion}, which we now recall. Let $\xi$ be a random variable. We state Jacod's criterion for the case where $\xi$ takes values in $\mathbb{R}^d$. In this subsection, $\mathbb H$ will always denote the initial expansion of $\mathbb F$ with $\xi$. Our results will be obtained with $\xi$ being either $\tau$ or $(\tau, X)$ and in both cases $\mathbb{H}$ indeed coincides with $\mathbb G^{\xi}$ after $\tau$ by Lemma~\ref{L:21} and Lemma~\ref{L:22}.

\begin{assumption}[Jacod's criterion]\label{A:J}
There exists a $\sigma$-finite measure $\eta$ on $\mathcal B(\mathbb{R}^d)$ such that $P(\xi \in \cdot \mid\mathcal F_t)(\omega) \ll \eta(\cdot)$ a.s.
\end{assumption}

Without loss of generality, $\eta$ may be chosen as the law of $\xi$. Under Assumption~\ref{A:J}, the $\mathcal F_t$-conditional density
$$
p_t(u;\omega) = \frac{P(\xi \in du \mid\mathcal F_t)(\omega)}{\eta(du)}
$$
exists, and can be chosen so that $(u,\omega,t)\mapsto p_t(u;\omega)$ is c\`adl\`ag in $t$ and measurable for the optional $\sigma$-field associated with the filtration $\widehat{\mathbb F}$ given by $\widehat{\mathcal F}_t=\cap_{u>t} \mathcal B(\mathbb{R}^d)\otimes \mathcal F_u$. See Lemma~1.8 in~\cite{Jacod:1987}. 

\begin{theorem}
Let $M$ be an $\mathbb F$~local martingale. 
\begin{itemize}
	\item[(i)] If $\tau$ satisfies Jacod's Criterion (Assumption~\ref{A:J}), then $M$ is a $\mathbb G^{\tau}$~semimartingale. 
	\item[(ii)] Let $X$ be a random variable such that $(\tau, X)$ satisfies Jacod's Criterion (Assumption~\ref{A:J}), then $M$ is a $\mathbb G^{(\tau, X)}$~semimartingale.
\end{itemize}
\end{theorem}

\begin{proof}
We prove (i). Let $\mathbb H^{\tau}$ be the initial expansion of F with $\tau$. It follows from Jacod's theorem (see Theorems VI.10 and VI.11 in [7]) that $M$ is an $\mathbb H^{\tau}$ semimartingale, which is $\mathbb G^{\tau}$ adapted. It is also a $\mathbb G^{\tau}$ semimartingale by Stricker's theorem. The proof of (ii) is similar.
\end{proof}

We provide the explicit decompositions using the following classical result by Jacod, see~\cite{Jacod:1987}, Theorem~2.5.

\begin{theorem}\label{T:J}
Let $M$ be an $\mathbb F$~local martingale, and assume Assumption~\ref{A:J} is satisfied. Then there exists set $B\in \mathcal B(\mathbb R^d)$, with $\eta(B)=0$, such that
\begin{itemize}
\item[(i)] $\langle p(u), M\rangle$ exists on $\{p_-(u)>0\}$ for every $u\notin B$,
\item[(ii)] there is an increasing predictable process $A$ and an $\widehat{\mathbb F}$~predictable function $k_t(u;\omega)$ such that for every $u\notin B$, $\langle p(u), M \rangle_t = \int_0^t k_s(u) p_{s-}(u) dA_s$ on $\{p_-(u)>0\}$,
\item[(iii)] $\int_0^t | k_s(\xi) | dA_s < \infty$ a.s.~for every $t\geq 0$ and $M_t - \int_0^t k_s(\xi) dA_s$ is an $\mathbb H$~local martingale.
\end{itemize}
\end{theorem}


Immediate consequences of Theorem~\ref{T:gen} and Theorem~\ref{T:J} are the following corollaries.

\begin{corollary}
Let $M$ be an $\mathbb F$~local martingale, and assume that $\tau$ satisfies Assumption~\ref{A:J}. Then $M$ is a $\mathbb G^{\tau}$~semimartingale, and
$$
M_t -  \int_0^{t\wedge\tau} \frac{d\langle M,\mu \rangle_s + dJ_s}{Z_{s-}} - \int_{t\wedge\tau}^t k_s(\tau)dA_s
$$
is the local martingale part of its $\mathbb G^{\tau}$~decomposition. Here $Z$, $\mu$ and $J$ are defined as in Theorem~\ref{T:JY}, and $k$ and $A$ as in Theorem~\ref{T:J} with $d=1$ and $\xi=\tau$.
\end{corollary}



Notice that this recovers the main result in~\cite{Jeanblanc/LeCam:2009} (Theorem 3.1), since by Theorem~\ref{T:J}~$(ii)$ we may write
$$
\int_{t\wedge\tau}^t k_s(\tau) dA_s =  \int_{t\wedge\tau}^t \frac{d\langle p(u),M\rangle_s}{p_{s-}(u)} \bigg|_{u=\tau},
$$
whenever the right side makes sense. See~\cite{Jeanblanc/LeCam:2009} for a detailed discussion.


\begin{corollary}
Let $M$ be an $\mathbb F$~local martingale. Let $X$ be a random variable and assume that $(\tau, X)$ satisfies Assumption~\ref{A:J}. Then $M$ is a $\mathbb G^{(\tau, X)}$~semimartingale, and
$$
M_t -  \int_0^{t\wedge\tau} \frac{d\langle M,\mu \rangle_s + dJ_s}{Z_{s-}} - \int_{t\wedge\tau}^t k_s(\tau,X)dA_s
$$
is the local martingale part of its $\mathbb G^{(\tau,X)}$~decomposition. Here $Z$, $\mu$ and $J$ are defined as in Theorem~\ref{T:JY}, and $k$ and $A$ as in Theorem~\ref{T:J} with $d=2$ and $\xi=(\tau,X)$.
\end{corollary}


\subsection{Absolute continuity of the Malliavin trace}

In two papers on models for insider trading in mathematical finance, Imkeller et al.~\cite{Imkeller_etal:2001} and Imkeller~\cite{Imkeller:2003} introduced an extension of Jacod's criterion for initial expansions, based on Malliavin calculus. Given a measure-valued random variable $F(du;\omega)$ defined on Wiener space with coordinate process $(W_t)_{0\leq t\leq 1}$, they introduce a Malliavin derivative $D_t F(du;\omega)$, defined for all $F$ satisfying certain regularity conditions. The full details are outside the scope of the present paper, and we refer the interested reader to~\cite{Imkeller_etal:2001} and~\cite{Imkeller:2003}. We continue to let $\mathbb H$ be the initial expansion of~$\mathbb F$.

The extension of Jacod's criterion is the following. Let $P_t(du,\omega)=P(\tau\in du\mid\mathcal F_t)(\omega)$, and assume that $D_tP_t(du,\omega)$ exists and satisfies
$$
\sup_{f \in C_b(\mathbb R),\ \|f\|\leq 1} E\left( \int_0^1 \langle D_s P_s(du), f\rangle^2 ds \right) < \infty.
$$
Here $C_b(\mathbb R)$ is the space of bounded and continuous functions on $\mathbb R$, $\|\cdot\|$ is the supremum norm, and $\langle F(du),f\rangle = \int_{\mathbb R_+} f(u) F(du)$ for a random measure $F$ and $f\in C_b(\mathbb R)$. Assume also that
$$
D_t P_t(du,\omega) \ll P_t(du,\omega) \quad a.s. \ \text{for\ all\ } t \in [0,1],
$$
and let $g_t(u;\omega)$ be a suitably measurable version of the corresponding density. Then they prove the following result.

\begin{theorem}\label{T:M}
Under the above conditions, if $\int_0^1 |g_t(\tau)| dt < \infty$ a.s., then
$$
W_t - \int_0^t g_s(\tau) ds
$$
is a Brownian motion in the initially expanded filtration $\mathbb H$.
\end{theorem}

One example where this holds but Jacod's criterion fails is $\tau=\sup_{0\leq t\leq 1}W_t$. In this case $g_t(\tau)$ can be computed explicitly and the $\mathbb H$~decomposition of $W$ obtained. Due to the martingale representation theorem in $\mathbb F$, this allows one to obtain the $\mathbb H$~decomposition for every $\mathbb F$~local martingale. Using Theorem~\ref{T:gen}, the decomposition in the progressively expanded filtration $\mathbb G$ can then also be obtained.

\begin{corollary}
Under the assumptions of Theorem~\ref{T:M},
$$
W_t - \int_0^{t\wedge\tau} \frac{d\langle W,Z\rangle_s}{Z_{s-}} - \int_{t\wedge\tau}^t g_s(\tau) ds
$$
is a $\mathbb G$~Brownian motion.
\end{corollary}

%
%

\section{Expansion with multiple random times} \label{S:2}

We now move on to progressive expansions with multiple random times. We start again with a filtered probability space $(\Omega,\mathcal F, \mathbb F,P)$, but instead of a single random time we consider a vector of random times
$$
\nv \tau = (\tau_1,\ldots,\tau_n).
$$
We emphasize that there are no restrictions on the ordering of the individual times. This is a significant departure from previous work in the field, where the times are customarily assumed to be ordered. The progressive expansion of $\mathbb F$ with $\nv \tau$ is
$$
\mathcal G_t = \bigcap_{u>t} \mathcal G^0_u  \qquad \text{where}\qquad  \mathcal G^0_t =\mathcal F_t \vee \sigma(\tau_i \wedge t; i=1,\ldots,n),
$$
and we are interested in the semimartingale decompositions of $\mathbb F$~local martingales in the $\mathbb G$~filtration. Several other filtrations will also appear, and we now introduce notation that will be in place for the remainder of this section, except in Theorem~\ref{T:dm11} and its corollary. Let $I\subset\{1,\ldots,n\}$ be an index set.

\begin{itemize}
\item $\sigma_I=\max_{i \in I} \tau_i$ and $\rho_I = \min_{j \notin I} \tau_j$
\item $\mathbb G^I$ denotes the initial expansion of $\mathbb F$ with the random vector $\nv\tau_I=(\tau_i)_{i\in I}$
\item $\mathbb H^I$ denotes the progressive expansion of $\mathbb G^I$ with the random time $\rho_I$
\end{itemize}

If $I=\emptyset$, then $\mathbb G^I = \mathbb F$ and $\mathbb H^I$ is the progressive expansion of $\mathbb F$ with $\rho_\emptyset=\min_{i=1,\ldots,n}\tau_i$. If on the other hand $I=\{1,\ldots,n\}$, then $\mathbb G^I = \mathbb H^I$, and coincides with the initial expansion of $\mathbb F$ with $\nv\tau_I=\nv \tau$. Notice also that we always have $\mathbb  G^I\subset \mathbb H^I$.

The idea from Section~\ref{S:1} can be modified to work in the present context. The intuition is that the filtrations $\mathbb G$ and $\mathbb H^I$ coincide on $\lc  \sigma_I, \rho_I\lc$. The $\mathbb G$ decomposition on $\lc  \sigma_I, \rho_I\lc$ of an $\mathbb F$~local martingale $M$ can then be obtained by computing its decomposition in $\mathbb H^I$. This is done in two steps. First it is obtained in $\mathbb G^I$ using, for instance, Jacod's theorem (Theorem~\ref{T:J}), and then in~$\mathbb H^I$ up to time $\rho_I$ using the Jeulin-Yor theorem (Theorem~\ref{T:JY}).

The following results collect some properties of the relationship between $\mathbb H^I$ and $\mathbb G$, thereby clarifying in which sense they coincide on $\lc \sigma_I,\rho_I\lc$. We take the index set $I$ to be given and fixed.

\begin{lemma} \label{L:HG}
Let $X$ be an $\mathcal H^I_t$-measurable random variable. Then the quantity $X\1{\sigma_I\leq t}$ is $\mathcal G_t$-measurable. As a consequence, if $H$ is an $\mathbb H^I$~optional (predictable) process, then $\nv 1_{\lc\sigma_I,\rho_I\lc}(H-H_{\sigma_I})$ is $\mathbb G$~optional (predictable).
\end{lemma}

\begin{proof} \label{L:alpha}
Let $X$ be of the form $X=f k(\rho_I\wedge t) \prod_{i\in I} h_i(\tau_i)$ for some $\mathcal F_t$-measurable random variable $f$ and Borel functions $k$, $h_i$. Then
$$
X\1{\sigma_I\leq t} = f k(\rho_I\wedge t) \prod_{i\in I} h_i(\tau_i\wedge t) \1{\sigma_I\leq t},
$$
which is $\mathcal G_t$-measurable, since $\sigma_I$ and $\rho_I$ are $\mathbb G$~stopping times and $\tau_i\wedge t$ is $\mathcal G_t$-measurable by construction. The Monotone Class theorem shows that the statement holds for every $X$ that is measurable for $\mathcal F_t\vee\sigma(\tau_i:i\in I)\vee\sigma(\rho_I\wedge t)$. $\mathbb H^I$ is the right-continuous version of this filtration, so the result follows. 

Now consider an $\mathbb H^I$~predictable process of the form $H=h\nv 1_{\rc s,t\rc}$ with $s\leq t$ and $h$ an $\mathcal H^I_s$-measurable random variable. Then
$$
H_r\nv 1_{\lc\sigma_I,\rho_I\lc}(r) = h\1{s<r\leq t}\1{\sigma_I\leq r < \rho_I} - h\1{s<\sigma_I\leq r<\rho_I},
$$
which defines a $\mathbb G$~predictable process using the first part of the lemma. An application of the Monotone Class theorem yields the desired result in the predictable case. The optional case is similar.
\end{proof}

\begin{lemma} \label{L:st}
Let $T_n$ be an $\mathbb H^I$~stopping time and define $T_n' = (\sigma_I\vee T_n)\wedge(\rho_I\vee n)$. Then $T'_n$ is a $\mathbb G$~stopping time.
\end{lemma}

\begin{proof}
We need to show that $\{T_n' > t\}\in\mathcal G_t$ for every $t\geq 0$. Careful set manipulations yield
\begin{align*}
\1{T_n' > t}
&= \1{\sigma_I\vee T_n > t} \1{\rho_I\vee n > t} \\
&= \left(1 - \1{\sigma_I\leq t}\1{T_n\leq t}\right) \left(1 - \1{\rho_I \leq t} \1{n \leq t}\right).
\end{align*}
We have that $1 - \1{\rho_I \leq t} \1{n \leq t} = 1 - \1{\rho_I \leq t}(1 - \1{n> t}) = \1{\rho_I>t} + \1{\rho_I\leq t}\1{n>t}$, so
\begin{align*}
\1{T_n' > t}
&= \1{\rho_I>t} - \1{\sigma_I \leq t} \1{T_n \leq t}\1{\rho_I>t} \\
& \qquad + \1{\rho_I\leq t}\1{n>t} - \1{\sigma_I\leq t}\1{T_n\leq t}\1{\rho_I\leq t}\1{n>t}.
\end{align*}
The first and third terms are $\mathcal G_t$-measurable since $\rho_I$ is a $\mathbb G$~stopping time. The second term is equal to $\1{T_n \leq t}\1{\sigma_I \leq t < \rho_I}$ and is thus $\mathcal G_t$-measurable by Lemma~\ref{L:HG}, which also gives the measurability of the fourth term.
\end{proof}

The next theorem generalizes Theorem~\ref{T:gen} to the case of progressive enlargement with multiple, not necessarily ranked times.

\begin{theorem}\label{T:dm1}
Let $M$ be an $\mathbb F$~local martingale such that $M_0=0$. For each $I\subset\{1,\ldots,n\}$, suppose there exists a $\mathbb G^I$~predictable finite variation process $A^I$ such that $M-A^I$ is a $\mathbb G^I$~local martingale. Moreover, define
$$
Z^I_t = P(\rho_I > t \mid \mathcal G^I_t),
$$
and let $\mu^I$ and $J^I$ be as in Theorem~\ref{T:JY}. Then $M$ is a $\mathbb G$~semimartingale with local martingale part $M_t - A_t$, where
$$
A_t = \sum_I \1{\sigma_I\leq \rho_I} \left( \int_{t\wedge\sigma_I}^{t\wedge\rho_I} dA^I_s + \int_{t\wedge\sigma_I}^{t\wedge\rho_I} \frac{d\langle M, \mu^I \rangle_s + dJ^I_s}{Z^I_{s-}} \right).
$$
Here the sum is taken over all $I\subset\{1,\ldots,n\}$.
\end{theorem}

Notice that when $I=\emptyset$, then $\sigma_I=0$ and $\mathbb G^I=\mathbb F$, so that $A^I=0$ and does not contribute to $A_t$. Similarly, when $I=\{1,\ldots,n\}$, $\rho_I=\infty$ and we have $Z^I_t=1$, causing both $\langle M, \mu^I \rangle$ and $J$ to vanish.

Before proving Theorem~\ref{T:dm1} we need the two following technical lemmas.

\begin{lemma}\label{L:sum}
Let $M$ be an $\mathcal{F}$-measurable process. Then
$$
\sum_{I\subset\{1,\ldots,n\}}\1{\sigma_I\leq\rho_I}(M_{t\wedge\rho_I}-M_{t\wedge\sigma_I})=M_t-M_0
$$
\end{lemma}

\begin{proof}
Let $\tau_{(0)}=0<\tau_{(1)}\leq\cdots\leq\tau_{(n)}$ be the ranked sequence of the random times $(\tau_i)_{1\leq i\leq n}$, and $\Sigma$ the permutation of $\{1,\ldots,n\}$ from which they are obtained, i.e $\tau_{\Sigma(i)}=\tau_{(i)}$, for all $i\in\{1,\ldots,n\}$. For each $0\leq j\leq n$, we introduce 
\begin{equation}\label{eq:S}
S_t^j=\sum_{I\mid |I|=j}\1{\sigma_I\leq\rho_I}(M_{t\wedge\rho_I}-M_{t\wedge\sigma_I}).
\end{equation}
We want to show that $\sum_{j=0}^{n}S_t^j=M_t-M_0$. First we claim that for $0\leq j\leq n$
$$
S^j_t= M_{t\wedge\tau_{(j+1)}}-M_{t\wedge\tau_{(j)}}.
$$
For $j=0$, this follows clearly since $\sigma_{\emptyset}=0$, $\rho_{\emptyset}=\tau_{(1)}$ and $\tau_{(0)}=0$. In order to prove the claim for $1\leq j\leq n$, we need to show that the only term that contributes to the sum in~(\ref{eq:S}) is the one corresponding to the subset $J=\{\Sigma(1),\ldots,\Sigma(j)\}$. Noticing that for this term $\sigma_J=\tau_{(j)}$ and $\rho_J=\tau_{(j+1)}$ will prove the claim. Fix $1\leq j\leq n$ and let $I$ be a subset of $\{1,\ldots,n\}$ such that $|I|=j$. Assume there exists $i$ such that $\Sigma(i)\notin I$ and $1\leq i\leq j$. Then $\tau_{\Sigma(i)}=\tau_{(i)}\geq \rho_{I}$. Also, since $|I|=j$, there exists $j+1\leq l\leq n$ such that $\Sigma(l)\in I$. It follows that $\sigma_{I}\geq \tau_{\Sigma(l)}=\tau_{(l)} \geq \tau_{(i)} \geq \rho_I$. This proves our claim. Now fix $0\leq k\leq n$. On $\{\tau_{(k)}\leq t<\tau_{(k+1)}\}$ (with the convention $\tau_{(n+1)}=\infty$), $S^j_t=M_{\tau_{(j+1)}}-M_{\tau_{(j)}}$ for each $j<k$, $S_t^k=M_t-M_{\tau_{(k)}}$ and $S_t^j=0$ for all $j>k$. Summing these terms finally yields
$$
\sum_{j=0}^{n}S^j_t=M_t-M_0.
$$
\end{proof}

\begin{lemma}\label{L:gl1}
Let $L$ be a local martingale in some filtration $\mathbb F$, suppose $\sigma$ and $\rho$ are two stopping times, and define a process $N_t = \1{\sigma\leq\rho} (L_{t\wedge\rho} - L_{t\wedge\sigma})$.
\begin{itemize}
\item[(i)] $N$ is again a local martingale.
\item[(ii)] Let $T$ be a stopping time and define $T' = (\sigma\vee T)\wedge(\rho\vee n)$ for a fixed $n$. Then $N_{t\wedge T} = N_{t\wedge T'}$.
\end{itemize}
\end{lemma}

\begin{proof}
\textit{Part $(i)$:} It is clear that $N$ is adapted, since it is null before $\sigma$. Next, by stopping we may assume $L$ is a uniformly integrable martingale. With $s<t$ we get
\begin{align*}
E\left(N_t \mid \mathcal F_s\right)
&= E\left( \1{s<\sigma \leq \rho} (L_{t\wedge\rho} - L_{t\wedge\sigma}) \mid \mathcal F_s\right) \\
& \qquad + E\left( \nv 1_{\{\sigma \leq s\}\cap\{\sigma \leq \rho\}} (L_{t\wedge\rho} - L_{t\wedge\sigma}) \mid \mathcal F_s\right).
\end{align*}
Since $\{\sigma \leq s\}\cap\{\sigma \leq \rho\}\in\mathcal F_s$, the second term equals $\nv 1_{\{\sigma \leq s\}\cap\{\sigma \leq \rho\}} (L_{s\wedge\rho} - L_{s\wedge\sigma})$. The first term equals
\begin{align*}
E\left\{ \1{s<\sigma \leq \rho} E\left( L_{t\wedge\rho} - L_{t\wedge\sigma} \mid \mathcal F_{s\vee\sigma}\right)  \mid \mathcal F_s\right\}
= E\left\{ \1{s<\sigma \leq \rho} ( L_{(s\vee\sigma)\wedge\rho} - L_{(s\vee\sigma)\wedge\sigma} )  \mid \mathcal F_s\right\},
\end{align*}
which is zero since $(s\vee\sigma)\wedge\rho = (s\vee\sigma)\wedge\sigma = \sigma$ on $\{s<\sigma \leq \rho\}$. Assembling the pieces yields
$$
E\left(N_t \mid \mathcal F_s\right) = \nv 1_{\{\sigma \leq s\}\cap\{\sigma \leq \rho\}} (L_{s\wedge\rho} - L_{s\wedge\sigma}) = \1{\sigma\leq\rho} (L_{s\wedge\rho} - L_{s\wedge\sigma}) = N_s,
$$
as desired.

\textit{Part $(ii)$:} The proof consists of a careful analysis of the interplay between the indicators involved in the definition of $N$ and $T'$. On $\{\sigma\leq\rho\}$ we have
$$
t\wedge T' \wedge\rho = t\wedge (\sigma\vee T)\wedge(\rho\vee n) \wedge\rho = t\wedge(\sigma\vee T)\wedge\rho
$$
and
$$
t\wedge T' \wedge\sigma = t\wedge (\sigma\vee T)\wedge(\rho\vee n) \wedge\sigma = t\wedge(\sigma\vee T)\wedge\sigma.
$$
On $\{T\leq \sigma\}\cap\{\sigma\leq\rho\}$, these are both equal to $t\wedge\sigma$, so on this set,
$$
N_{t\wedge T'}\1{T \leq \sigma}=\1{\sigma\leq\rho}\1{T \leq \sigma} (L_{t\wedge\sigma} - L_{t\wedge\sigma}) = 0.
$$
Hence $N_{t\wedge T'}=N_{t\wedge T'}\1{T > \sigma}$. Similarly, $N_{t\wedge T}=0$ on $\{T\leq\sigma\}$, and hence $N_{t\wedge T}\1{T > \sigma}=N_{t\wedge T}$. But on $\{T>\sigma\}\cap\{\sigma\leq\rho\}$,
$$
t\wedge(\sigma\vee T')\wedge\rho = t\wedge T\wedge\rho \qquad\text{and}\qquad t\wedge(\sigma\vee T')\wedge\sigma = t\wedge T\wedge\sigma,
$$
so $N_{t\wedge T'}\1{T > \sigma} = N_{t\wedge T}\1{T > \sigma}$. This yields $N_{t\wedge T'}=N_{t\wedge T}$ as required.
\end{proof}

\begin{proof}[Proof of Theorem~\ref{T:dm1}]
For each fixed index set $I$, the process $M_t-\int_0^t dA^I_s$ is a local martingale in the initially expanded filtration $\mathbb G^I$ by assumption. Now, $\mathbb H^I$ is obtained from $\mathbb G^I$ by a progressive expansion with $\rho_I$, so Theorem~\ref{T:JY} yields that
$$
M^I_t = M_{t\wedge \rho_I}-\int_0^{t\wedge \rho_I} dA^I_s - \int_0^{t\wedge \rho_I} \frac{d\langle M,\mu^I\rangle_s + dJ^I_s}{Z^I_{s-}}
$$
is an $\mathbb H^I$~local martingale. Define the process
$$
N^I_t = \1{\sigma_I \leq \rho_I}\left( M^I_{t\wedge\rho_I} - M^I_{t\wedge\sigma_I}\right).
$$
Our goal is to prove that $N^I$ is a $\mathbb G$~local martingale. This will imply the statement of the theorem, since summing the $N^I$ over all index sets $I$ and using Lemma \ref{L:sum} yields precisely $M-A$.

By part~$(i)$ of Lemma~\ref{L:gl1}, $N^I$ is a local martingale in $\mathbb H^I$. Write
$$
N^I_t = \1{\sigma_I\leq t <\rho_I}\left(M^I_t-M^I_{\sigma_I}\right) +  \1{\sigma_I\leq \rho_I < t}\left(M^I_{\rho_I}-M^I_{\sigma_I}\right) = Y_1 + Y_2
$$
and apply Lemma~\ref{L:HG} with $X=Y_1$ for the first term and $X=Y_2$ for the second to see that $N^I$ is in fact $\mathbb G$~adapted. The use of Lemma~\ref{L:HG} is valid because both $Y_1$ and $Y_2$ are $\mathcal H^I_t$-measurable.

Next, let $(T_n)_{n\geq 1}$ be a reducing sequence for $N^I$ in $\mathbb H^I$. By Lemma~\ref{L:st} we know that $T_n'=(\sigma_I\vee T_n)\wedge(\rho_I\vee n)$ are $\mathbb G$~stopping times, and since $T'_n\geq T_n\wedge n$ we have $T'_n\uparrow\infty$ a.s. Moreover, part~$(ii)$ of Lemma~\ref{L:gl1} implies that $N^I_{t\wedge T_n} = N^I_{t\wedge T_n'}$. Hence $(N^I_{t\wedge T_n'})_{t\geq 0}$ is an $\mathbb H^I$~martingale that is $\mathbb G$~adapted, and therefore even a $\mathbb G$~martingale. We deduce from the above that $N^I$ is a $\mathbb G$~local martingale. A final application of Lemma~\ref{L:HG} shows that $A$ is $\mathbb G$~predictable, so we obtain indeed the $\mathbb G$~semimartingale decomposition of $M$. This completes the proof of Theorem~\ref{T:dm1}.
\end{proof}

We now proceed to study the special case where the vector $\nv\tau$ of random times satisfies Jacod's criterion, meaning that Assumption~\ref{A:J} holds, now with state space $E=\mathbb R^n_+$. Again there is no loss of generality to let $\eta$ be the law of $\nv \tau$. We will further assume that $\eta$ is absolutely continuous w.r.t to Lebesgue measure, so that $\eta(d\nv u)=h(\nv u)d\nv u$. Provided the law of $\nv \tau$ does not have atoms, this restriction is not essential---everything that follows goes through without it---but it simplifies the already quite cumbersome notation.

The joint $\mathcal F_t$~conditional density corresponding to this choice of $\eta$ is denoted $p_t(\nv u;\omega)$. That is,
$$
P( \nv \tau \in d\nv u \mid \mathcal F_t) = p_t(\nv u)du_1\cdots du_n,
$$
where we suppressed the dependence on $\omega$. Now, for an index set $I\subset\{1,\ldots,n\}$ with $|I|=m$, we have, for $\nv u_I\in\mathbb R_+^m$,
$$
P( \nv\tau_I \leq \nv u_I \mid\mathcal F_t) =  \int_{\nv v_I \leq \nv u_I} \int_{\nv v_{-I}\geq 0} p_t(\nv v_I; \nv v_{-I}) d\nv v_{-I}d\nv v_I,
$$
where $\nv\tau_I$ is the subvector of $\nv \tau$ whose components have indices in $I$, and where $\nv v_I$ and $\nv v_{-I}$ are the subvectors of $\nv v$ with indices in $I$, respectively not in $I$. Inequalities should be interpreted componentwise. The above shows that $\nv\tau_I$ also satisfies Assumption~\ref{A:J}, so that there is an appropriately measurable function $p^I_t(\nv u_I;\omega)$ such that
$$
P( \nv \tau_I \in d\nv u_I \mid \mathcal F_t) = p^I_t(\nv u_I)d\nv u_I.
$$
Moreover, this conditional density $p^I$ is given by
$$
p^I_t(\nv u_I) =  \int_{\mathbb R^{n-m}_+} p_t(\nv u_I; \nv u_{-I}) d\nv u_{-I}.
$$

Define
$$
p_t(\nv u_{-I} \mid \nv \tau_I) = \frac{p_t(\nv\tau_I; \nv u_{-I})}{\int_0^\infty\cdots\int_0^\infty p_t(\nv\tau_I; \nv v_{-I})d\nv v_{-I}}.
$$
This is the conditional density of $\nv\tau_{-I}$ given $\mathcal F_t\vee\sigma(\tau_I)$, as one can verify using standard arguments. Defining
\begin{equation}\label{E:ZI}
Z^I_t = \int_t^\infty\cdots\int_t^\infty p_t(\nv u_{-I} \mid \nv \tau_I) d\nv u_{-I},
\end{equation}
we have $Z^I_t = P(\rho_I > t \mid \mathcal F_t\vee\sigma(\tau_I) )$. One then readily checks that we also have
$$
Z^I_t = P(\rho_I > t \mid \mathcal G^I_t).
$$

We can now state the decomposition theorem for continuous $\mathbb F$~local martingales in the progressively expanded filtration~$\mathbb G$, when Jacod's criterion is satisfied. Recall that $\sigma_I=\max_{i \in I} \tau_i$ and $\rho_I = \min_{j \notin I} \tau_j$.

\begin{corollary}
Let $M$ be an $\mathbb F$~local martingale, and assume Assumption~\ref{A:J} is satisfied for $\nv\tau=(\tau_1,\ldots,\tau_n)$. Then $M$ is a $\mathbb{G}$~semimartingale. Furthermore, assume that $M$ is continuous. For each $I\subset\{1,\ldots,n\}$, let $Z^I$ be given by~(\ref{E:ZI}) and let $\mu^I$ and $J^I$ be as in Theorem~\ref{T:JY}. Then $M_t - A_t$ is a $\mathbb G$~local martingale, where
$$
A_t = \sum_I \1{\sigma_I\leq \rho_I}\left( \int_{t\wedge\sigma_I}^{t\wedge\rho_I} \frac{d\langle p^I(\nv u_I),M\rangle_s}{p^I_{s-}(\nv u_I)}\bigg|_{\nv u_I=\nv\tau_I} + \int_{t\wedge\sigma_I}^{t\wedge\rho_I} \frac{d\langle M, \mu^I\rangle_s+dJ^I_s}{Z^I_{s-}} \right).
$$
Here the sum is taken over all $I\subset\{1,\ldots,n\}$.
\end{corollary}

\begin{proof}
We apply Theorem~\ref{T:dm1} and notice that it follows from Theorem~\ref{T:J} that
$$
A^I_t = \int_0^t k^I_s(\nv\tau_I) d\langle M,M\rangle_s = \int_0^t \frac{d\langle p^I(\nv u_I),M\rangle_s}{p^I_{s-}(\nv u_I)}\bigg|_{\nv u_I=\nv\tau_I}.
$$
This completes the proof.
\end{proof}

We end this section by pointing out that the filtration $\mathbb G^{(\tau,X)}$ introduced in Section~\ref{S:1} can be generalized to the multiple time case. To state the precise result, we first suppose that each random time $\tau_i$ is accompanied by a random variable $X_i$, and let $\nv X=(X_1,\ldots,X_n)$. Define the filtration $\mathbb G^{(\nv \tau, \nv X)}$ by
$$
\mathcal G^{(\nv\tau,\nv X)}_t = \bigcap_{u>t} \mathcal G^{0,(\nv \tau,\nv X)}_u,
$$
where
$$
\mathcal G^{0,(\nv \tau,\nv X)}_t =\mathcal F_t \vee \sigma(\tau_i \wedge t:i=1,\ldots,n) \vee \sigma(X_i1_{\{\tau_i\leq t\}}:i=1,\ldots,n).
$$
Let $I\subset\{1,\ldots,n\}$ be an index set. Assume for simplicity that $P(\tau_i=\tau_j)=0$ for $i\neq j$. We may then define
\begin{itemize}
\item $\nv X_I = (X_i)_{i\in I}$.
\item $Y_I = X_{i^*}$, where $i^*\in I$ is the index for which $\rho_I = \tau_{i^*}$.
\end{itemize}
For the statement and proof of Theorem~\ref{T:dm11}, we redefine the objects $\mathbb G^I$ and $\mathbb H^I$ as follows. For an index set $I\subset\{1,\ldots,n\}$,
\begin{itemize}
\item $\mathbb G^I$ denotes the initial expansion of $\mathbb F$ with the random vector $(\nv\tau_I,\nv X_I)=(\tau_i,X_i)_{i\in I}$.
\item $\mathbb H^I$ denotes the $(\rho_I,Y_I)$-expansion of $\mathbb G^I$.
\end{itemize}

\begin{theorem}\label{T:dm11}
Let $M$ be an $\mathbb F$~local martingale such that $M_0=0$. For each $I\subset\{1,\ldots,n\}$, suppose there exists a $\mathbb G^I$~predictable finite variation process $A^I$ such that $M-A^I$ is a $\mathbb G^I$~local martingale. Moreover, define
$$
Z^I_t = P(\rho_I > t \mid \mathcal G^I_t),
$$
and let $\mu^I$ and $J^I$ be as in Theorem~\ref{T:JY}. Then $M$ is a $\mathbb G^{(\nv \tau, \nv X)}$~semimartingale with local martingale part $M_t - A_t$, where
$$
A_t = \sum_I \1{\sigma_I\leq \rho_I} \left( \int_{t\wedge\sigma_I}^{t\wedge\rho_I} dA^I_s + \int_{t\wedge\sigma_I}^{t\wedge\rho_I} \frac{d\langle M, \mu^I \rangle_s + dJ^I_s}{Z^I_{s-}} \right).
$$
Here the sum is taken over all $I\subset\{1,\ldots,n\}$.
\end{theorem}

\begin{proof}
The proof is the same as that of Theorem~\ref{T:dm1}, except for the following points: Instead of Theorem~\ref{T:JY}, we invoke Theorem~\ref{T:JYExt}, which is justified by Lemma~\ref{L:inters}. Moreover, it must be verified that Lemma~\ref{L:HG} remains valid for the redefined $\mathbb H^I$ and $\mathbb G=\mathbb G^{(\nv \tau, \nv X)}$. This is easily done: in the proof of Lemma~\ref{L:HG}, simply replace $X=f k(\rho_I\wedge t) \prod_{i\in I} h_i(\tau_i)$ by
$$
X=f k(\rho_I\wedge t) \ell(Y_I \1{\rho_I\leq t} ) \prod_{i\in I} h_i(\tau_i) g_i(X_i\1{\tau_i\leq t}),
$$
where $\ell(\cdot)$ and $g_i(\cdot)$ are Borel functions.
\end{proof}

Define the process
$$
N^n_t=\sum_{i=1}^nX_i1_{\{\tau_i\leq t\}}
$$
Let $\mathbb{N}^n$ be the smallest right continuous filtration containing $\mathbb{F}$ and to which the process $N^n$ is adapted. Under the same assumption as in Theorem~\ref{T:dm11}, $\mathbb{F}$~semimartingales remain $\mathbb{N}^n$~semimartingales.

\begin{corollary}
Let $M$ be an $\mathbb F$~local martingale such that $M_0=0$. For each $I\subset\{1,\ldots,n\}$, suppose there exists a $\mathbb G^I$~predictable finite variation process $A^I$ such that $M-A^I$ is a $\mathbb G^I$~local martingale. Then $M$ is a $\mathbb{N}^n$~semimartingale.
\end{corollary}

\begin{proof}
Applying Theorem~\ref{T:dm11}, $M$ is a $\mathbb{G}^{(\nv \tau, \nv X)}$~semimartingale. Since $\mathbb{N}^n\subset\mathbb{G}^{(\nv \tau, \nv X)}$ and $M$ is adapted to $\mathbb{N}^n$, it follows from Stricker's theorem (\cite{Protter:2005} Theorem II.4) that $M$ is a $\mathbb{N}^n$~semimartingale.
\end{proof}

\section{Connection to filtration shrinkage} \label{S:shr}

It has been observed that the optional projection of a local martingale $M$ onto a filtration to which it is not adapted may lose the local martingale property, see~\cite{Follmer/Protter:2010}. However, general conditions for when this happens are not available. In this section we give a partial result in this direction. We prove that when $\mathbb F$ is a Brownian filtration, Jacod's criterion together with additional regularity of the conditional densities guarantee that for any sufficiently regular $\mathbb F$~local martingale $M$, $^oM^{\mathbb H}$ is a local martingale in $\mathbb G$, where $M^{\mathbb H}$ is the local martingale part of $M$ in $\mathbb H$ and $^oX$ denotes optional projection onto $\mathbb G$. For a locally integrable finite variation process $A$, its dual predictable projection onto $\mathbb G$ will be denoted $A^p$.

\begin{theorem}\label{T:30}
Let $M$ be an $\mathbb F$~local martingale and let $\mathbb H$ be a filtration that coincides with $\mathbb G$ after $\tau$. Suppose $M$ is an $\mathbb H$~semimartingale with decomposition
$$
M_t = M^{\mathbb H}_t + A^{\mathbb H}_t.
$$
Assume $A^{\mathbb H}$ has an optional projection $^oA^{\mathbb H}$ onto $\mathbb G$ which is $\mathbb G$~locally integrable. Then $^oM^{\mathbb H}$ exists. It is a $\mathbb G$~local martingale if and only if
$$
(A^{\mathbb H}_{\cdot\wedge\tau})^p_t = \int_0^{t\wedge\tau} \frac{d\langle M,\mu\rangle_s + dJ_s}{Z_{s-}} \ \text{a.s.\ for\ all\ } t\geq 0,
$$
where $Z$, $\mu$ and $J$ are as in Theorem~\ref{T:JY}.
\end{theorem}

\begin{proof}
Since $M$ is $\mathbb G$~adapted and $^oA^{\mathbb H}$ exists, $^oM^{\mathbb H}$ also exists. Using Lemma~\ref{L:20}~$(ii)$ we get
$$
^oM^{\mathbb H}_t = M_t -\ ^o(A^{\mathbb H}_{\cdot\wedge\tau})_t - \int_{t\wedge\tau}^t dA^{\mathbb H}_s.
$$
Moreover, since $^oA^{\mathbb H}$ exists and is $\mathbb G$~locally integrable, the same holds for $^o(A^{\mathbb H}_{\cdot\wedge\tau})$, which therefore has a compensator which coincides with the dual predictable projection $(A^{\mathbb H}_{\cdot\wedge\tau})^p$. Let $\tilde M_t =\ ^o(A^{\mathbb H}_{\cdot\wedge\tau}) - (A^{\mathbb H}_{\cdot\wedge\tau})^p$, which is a $\mathbb G$~local martingale. Then
$$
^oM^{\mathbb H}_t = M_t - (A^{\mathbb H}_{\cdot\wedge\tau})^p_t - \tilde M_t - \int_{t\wedge\tau}^t dA^{\mathbb H}_s.
$$
Since $M$ is an $\mathbb H$~semimartingale that is $\mathbb G$~adapted, it is also a $\mathbb G$~semimartingale, see Theorem~II.4 in Protter~\cite{Protter:2005}. Let its decomposition be given by
$$
M_t = M^{\mathbb G}_t + A^{\mathbb G}_t = M^{\mathbb G}_t + A^{\mathbb G}_{t\wedge\tau} + \int_{t\wedge\tau}^t dA^{\mathbb G}_s.
$$
By Theorem~\ref{T:gen}, we have $\int_{t\wedge\tau}^t dA^{\mathbb G}_s = \int_{t\wedge\tau}^t dA^{\mathbb H}_s$, so that
$$
^oM^{\mathbb H}_t  = \Big( M^{\mathbb G}_t  - \tilde M_t \Big)  +  \Big( A^{\mathbb G}_{t\wedge\tau} - (A^{\mathbb H}_{\cdot\wedge\tau})^p_t \Big).
$$
By Theorem~\ref{T:JY} we have
$$
A^{\mathbb G}_{t\wedge\tau} = \int_0^{t\wedge\tau} \frac{d\langle M,\mu\rangle_s + dJ_s}{Z_{s-}},
$$
so the result follows since the predictable, finite variation process $A^{\mathbb G}_{t\wedge\tau} - (A^{\mathbb H}_{\cdot\wedge\tau})^p_t$ is zero for all $t$ a.s.~if and only if $^oM^{\mathbb H}$ is a $\mathbb G$~local martingale.
\end{proof}

We now specialize to the case where $\mathbb F$ is generated by a Brownian motion $W$, and $\tau$ satisfies Jacod's Criterion (Assumption~\ref{A:J}). By martingale representation, the $\mathcal F_t$-conditional densities $p_t(u)$ satisfy
\begin{equation}\label{e2}
p_t(u) = p_0(u) + \int_0^t q_s(u)dW_s
\end{equation}
for some family $\{q(u): u\geq 0\}$ of $\mathbb F$~predictable processes $(q_t(u))_{t\geq 0}$. In order to prove the main result of this section, we need the following lemma.

\begin{lemma}\label{L:31}
Let $\mathbb G$ and $\mathbb H$ be two filtrations such that $\mathbb G\subset\mathbb H$. Suppose $a$ is an $\mathbb{H}$-adapted process such that $E\{ \int_0^t |a_s| ds\} < \infty$ for each $t\geq 0$. Then
$$
M_t = E\Big(\int_0^ta_sds\mid\mathcal{G}_t\Big)-\int_0^t E(a_s\mid\mathcal{G}_s)ds
$$
is a $\mathbb{G}$~martingale.
\end{lemma}
 
\begin{proof}
Let $0\leq s<t$. Then
\begin{align*}
E(M_{t}\mid\mathcal{G}_s)&=E\Big(E(\int_0^{t} a_udu\mid\mathcal{G}_{t})-\int_0^{t}E(a_u\mid\mathcal{G}_u)du\mid\mathcal{G}_s\Big)\\
&=E(\int_0^{t}a_udu\mid\mathcal{G}_s)-E(\int_0^{t}E(a_u\mid\mathcal{G}_u)du\mid\mathcal{G}_s)\\
&=E(\int_0^sa_udu\mid\mathcal{G}_s)+E(\int_s^{t}a_udu\mid\mathcal{G}_s)-\int_0^sE(a_u\mid\mathcal{G}_u)du-\mathbb{E}(\int_s^{t}\mathbb{E}(a_u\mid\mathcal{G}_u)du\mid\mathcal{G}_s).
\end{align*}
By Fubini's theorem we obtain $E(\int_s^{t}E(a_u\mid\mathcal{G}_u)du\mid\mathcal{G}_s)=\int_s^{t}E(a_u\mid\mathcal{G}_s)du$ and $E(\int_s^{t}a_udu\mid\mathcal{G}_s)=\int_s^{t}E(a_u\mid\mathcal{G}_s)du$. Hence
$$
E(M_{t}\mid\mathcal{G}_s)=E\Big(\int_0^{t}a_udu\mid\mathcal{G}_u\Big)-\int_0^sE(a_u\mid\mathcal{G}_u)du=M_s,
$$
as required.
\end{proof}

We may now state the following theorem.

\begin{theorem}
Assume the processes $q(u)$ given in~\eqref{e2} are bounded by some integrable process. Let $M$ be an $\mathbb F$~local martingale and $k_t(u)$ be as in Theorem~\ref{T:J}. If
$$
E\Big( \int_0^t \big| k_s(\tau) \big| d\langle M,M\rangle_s \Big) < \infty \quad \text{for\ all\ } t\geq 0,
$$
then $^oM^{\mathbb H}$ is a $\mathbb G$~local martingale, where $M^{\mathbb H}$ is the local martingale part of the $\mathbb H$~decomposition of $M$.
\end{theorem}

\begin{proof}
Note that $M$ is indeed an $\mathbb H$~semimartingale by Theorem~\ref{T:J}, and that its finite variation part $A^{\mathbb H}$ is locally integrable and given by
$$
A^{\mathbb H}_t =  \int_0^t  k_s(\tau) d\langle M,M\rangle_s.
$$
By Theorem~\ref{T:30} we must show that $(A^{\mathbb H}_{\cdot\wedge\tau})^p$ and $\int_0^{t\wedge\tau} \frac{d\langle M,Z\rangle_s}{Z_{s-}}$ are equal, so we compute these two quantities. We start with the first one. By martingale representation, there is a predictable process $(m_t)_{t\geq 0}$ such that $M_t = M_0 + \int_0^t m_sdW_s$. Thus $d\langle M,M\rangle_s=m^2_s ds$, and using Lemma~\ref{L:31} we get
$$
^o(A^{\mathbb H}_{\cdot\wedge\tau})_t = E\Big( \int_0^t \1{\tau> s} k_s(\tau) m^2_s ds \mid \mathcal G_t\Big) = N_t + \int_0^t \1{\tau> s} E\left( k_s(\tau)  \mid \mathcal G_s \right) m^2_s ds,
$$
where $N$ is a $\mathbb G$~martingale. Hence $(A^{\mathbb H}_{\cdot\wedge\tau})^p$ is given by the second term on the right side, so we focus on this term. We have
$$
 \1{\tau> s}E( k_s(\tau) \mid\mathcal G_s) =  \1{\tau> s}\frac{\int_s^\infty k_s(u) p_s(u) du}{\int_s^\infty p_s(u) du} =  \1{\tau> s}\frac{1}{Z_{s-}} \int_s^\infty k_s(u)p_s(u)du,
$$
where the first equality holds because the conditional density of $\tau$ given $\mathcal F_t\vee\sigma(\{\tau>s\})$ is $k_s(u) / \int_s^\infty p_s(u)du$. Using part~$(i)$ of Theorem~\ref{T:J} we then get
\begin{equation}\label{Eq:40}
(A^{\mathbb H}_{\cdot\wedge\tau})^p_t = \int_0^{t\wedge\tau} \frac{1}{Z_{s-}}\left(  \int_s^\infty k_s(u)p_s(u)du \right) m^2_s ds =  \int_0^{t\wedge\tau} \frac{1}{Z_{s-}}\left(  \int_s^\infty q_s(u) du \right) m_s ds.
\end{equation}

We now deal with $\int_0^{t\wedge\tau} \frac{d\langle M,Z\rangle_s}{Z_{s-}}$. Since $Z_t = \int_t^\infty p_t(u) du = 1 - \int_0^t p_t(u)du$, we have
$$
Z_t + \int_0^t p_u(u)du = 1 + \int_0^t (p_u(u) - p_t(u) )du = 1 - \int_0^t \int_u^t q_s(u)dW_s du.
$$
Since all $q(u)$ are bounded by an integrable process, we may exchange the integrals in the last term. This yields
$$
Z_t + \int_0^t p_u(u)du = 1 - \int_0^t \int_0^u q_s(u) du dW_s.
$$
Hence $d\langle Z,M\rangle_s = -m_s ( \int_0^sq_s(u)du)ds$. Now, we claim that $\int_0^\infty q_s(u)du=0$ for all~$s$ a.s. This follows from
$$
0=\int_0^\infty (p_t(u) - p_0(u))du = \int_0^\infty\int_0^t q_s(u)dW_s du = \int_0^t\int_0^\infty q_s(u)dudW_s,
$$
where the first equality holds because $\int_0^{\infty}p_t(u)du=1$ for every $t\geq 0$ a.s.
Thus $d\langle Z,M\rangle_s = m_s ( \int_s^\infty q_s(u)du)ds$, so
$$
\int_0^{t\wedge\tau} \frac{d\langle M,Z\rangle_s}{Z_{s-}} = \int_0^{t\wedge\tau} \frac{1}{Z_{s-}} m_s \left( \int_s^\infty q_s(u)du \right)ds,
$$
which coincides with the previously established expression for $(A^{\mathbb H}_{\cdot\wedge\tau})^p_t$ given in Equation~(\ref{Eq:40}). The claim now follows from Theorem~\ref{T:30}.
\end{proof}


\begin{thebibliography}{99}

\bibitem{Follmer/Protter:2010} H. F\"{o}llmer and P. Protter. Local martingales and filtration shrinkage.  \emph{ESAIM: Probability and Statistics}, Available on CJO 2006 doi:10.1051/ps/2010023, 2011.

\bibitem{Imkeller:2003} P. Imkeller. Malliavin�s calculus in insider models: additional utility and free lunches. \emph{Mathematical Finance}, {\bf 13}, 153-169, 2003.

\bibitem{Imkeller_etal:2001} P. Imkeller, M. Pontier, and F. Weisz. Free lunch and arbitrage possibilities in a financial market model with an insider, \emph{Stoch. Proc. Appl.}, {\bf 92}, 103-130, 2001.

\bibitem{Jacod:1987} J. Jacod. \textit{Grossissement initial, hypoth\`ese (H') et th\'eor\`eme de Girsanov}, Volume 1118 of Lecture Notes in Mathematics, 15-35; Springer-Verlag, 1987.

\bibitem{Jeanblanc/LeCam:2009} M. Jeanblanc and Y. Le Cam. Progressive enlargement of filtrations with initial times. \emph{Stoch. Proc. Appl.}, {\bf 119}, 2523-2543, 2009.

\bibitem{Jeulin/Yor:1985} T. Jeulin and M. Yor, editors. \textit{Grossissement de filtrations: examples et applications.} Volume 1118 of Lecture Notes in Mathematics. Springer, Berlin, 1985.

\bibitem{Protter:2005} P. Protter. \emph{Stochastic Integration and Differential Equations}, Springer-Verlag, Heidelberg, second edition, 2005.

\end{thebibliography}
\end{document}